\documentclass[letterpaper,reqno]{amsart}

\usepackage{amsmath} 
\usepackage{amsfonts} 
\usepackage{amssymb} 
\usepackage{epsfig} 
\usepackage{graphicx} 
\usepackage{multirow} 
\usepackage{verbatim} 
\usepackage{rotating} 
\usepackage[all]{xy} 
\usepackage{appendix}

\newtheorem{theorem}{Theorem}[section]
\newtheorem{lemma}[theorem]{Lemma}

\theoremstyle{definition}
\newtheorem{definition}[theorem]{Definition}
\newtheorem{proposition}[theorem]{Proposition}

\newtheorem{warning}[theorem]{Warning}

\theoremstyle{remark}
\newtheorem{remark}[theorem]{Remark}
\theoremstyle{notation}
\newtheorem{notation}[theorem]{Notation}
\numberwithin{equation}{section}
\theoremstyle{corollary}
\newtheorem{corollary}[theorem]{Corollary}

\usepackage[bookmarks=false]{hyperref}
\newcommand{\Map}{\mathrm{Map}}

\newcommand{\Set}{\mathsf{Set}}

\newcommand{\C}{\mathsf{C}}

\newcommand{\EE}{\mathsf{E}_{\infty}}

\newcommand{\HH}{\mathrm{HH}_{k}}

\newcommand{\D}{\mathsf{D}}

\newcommand{\coh}{\mathrm{H}}

\newcommand{\Ho}{\mathrm{Ho}}

\newcommand{\Algk}{\mathsf{dgAlg}_{k}}

\newcommand{\sSetp}{\mathsf{sSet}_{\ast}}
\newcommand{\sSet}{\mathsf{sSet}}

\newcommand{\CAlgk}{\mathsf{dgCAlg}_{k}}

\newcommand{\holim}{\mathrm{holim}}
\newcommand{\hocolim}{\mathrm{hocolim}}

\begin{document}

\title[]{ rational homotopy theory of function spaces and Hochschild cohomology}

\author[]{Ilias Amrani}
\address{Department of Mathematics, Masaryk University\\ Kotlarska 2\\ Brno, Czech Republic.}
\email{ilias.amranifedotov@gmail.com}
\email{amrani@math.muni.cz}
\thanks{Partially supported by the project CZ.1.07/2.3.00/20.0003
of the Operational Programme Education for Competitiveness of the Ministry
of Education, Youth and Sports of the Czech Republic.
}

\thanks{}

\subjclass[2000]{Primary 55, Secondary 14 , 16, 18}




\begin{abstract}

Given a map $f: X\rightarrow Y$ of simply connected spaces of finite type such. The space of based loops at $f$ of the space of maps between $X$ and $Y$ is denoted by $\Omega_{f}\Map(X,Y)$. For $n> 0$, we give a model categorical interpretation of  the existence (in functorial way) of an injective map of $\mathbb{Q}$-vector spaces $\pi_{n} \Omega_{f}\Map(X,Y_{\mathbb{Q}})  \rightarrow \HH^{-n}(C^{\ast}(Y),C^{\ast}(X)_{f})$, where $\HH^{\ast}$ is the (negative) Hochschild cohomology and $C^{\ast}(X)_{f}$ is the rational cochain complex associated to $X$ equipped with a structure of $C^{\ast}(Y)$-differential graded bimodule via the induced map of differential graded algebras $f^{\ast}:  C^{\ast}(Y)\rightarrow C^{\ast}(X)$. Moreover, we identifiy the image in presice way by using the Hodge filtration on Hochschild cohomology. In particular, when $X=Y$, we describe the fundamental group of the identity component of the monoid of self equivalence of a (rationalization of) space $X$ i.e., $\pi_{1} Aut(X_{\mathbb{Q}})_{id}$.
\end{abstract}

\maketitle
\section*{Introduction}
Our main goal in this article is the study of the function space $\Map(X,Y)$ between two \textbf{rational} topological spaces from non-commutative point of view. More precisely, for a fixed map $f:X\rightarrow Y$ we study the homotopy groups of the path connected component $\Map(X,Y_{\mathbb{Q}})_{f}$. It is well known \cite{block2005andre} that \textbf{rationally} (under some finiteness conditions) the homotopy groups $(\pi_{n}, ~n> 1)$ of  $\Map(X,Y_{\mathbb{Q}})_{f}$ are given by the Andr\'e-Quillen cohomology $\mathrm{AQ}^{-n}(C^{\ast}(Y),C^{\ast}(X))$, where $C^{\ast}(X)$ is seen as a module over $C^{\ast}(Y)$ via the induced map of $\EE$-differential graded algebras $f^{\ast}:C^{\ast}(Y)\rightarrow C^{\ast}(X)$. The point is that the Andr\'e-Quillen cohomology is quite complicated to compute. We should notice that we are using the fact that any rational $\EE$-differential graded algebra is equivalent to a rational commutative differential graded algebra. Let $k$ be any commutative ring, and denote the model category of $\EE$-differential graded $k$-algebras by $\EE-\Algk$ and the model category of associative differential graded $k$-algebras by $\Algk$. The (derived) forgetful functor  $U:\EE-\Algk\rightarrow \Algk$ induces a map of simplicial sets 
$$\alpha:\Map_{\EE-\Algk}(R,S)\rightarrow \Map_{\Algk}(R,S):=\Map_{\Algk}(UR,US).$$
In all what follows, we will consider only the positively graded algebras with increasing differentials by degree one. A perfect example is the cochain complex associated to a topological space. The interpretation of the higher homotopy groups is quite simple, in fact in \cite{amrani2013mapping}, we have shown that for a given map $f:R\rightarrow S$ of differential graded $k$-algebras we have
\begin{equation}\label{eq}
\pi_{n}\Map_{\Algk}(R,S)_{f}\cong \HH^{-n+1}(R,S_{f}) ~ \forall~n>1,
\end{equation}
where $\HH^{\ast}$ is the Hochschild cohomology and $S$ is seen as $R$-bimodule via $f$. 
\subsection*{Rational homotopy theory} When $k=\mathbb{Q}$, Sullivan has proven that there is an $\infty$-equivalence between the category of simply connected rational spaces (finite type) and a subcategory of simply connected commutative differential graded $k$-algebras (of finite type)  \cite{sullivan1977infinitesimal}. The $\infty$-equivalence is given by the cochain functor $C^{\ast}(-,k)$ after strictification. For any map $f:X\rightarrow Y$ of  rational simply connected spaces, the forgetful functor $U$ induces the following map of $k$-vector spaces
$$\pi_{n+1}\alpha:~\pi_{n+1}\Map(X,Y_{\mathbb{Q}})_{f}\rightarrow \HH^{-n}(C^{\ast}(Y), C^{\ast}(X)_{f}).$$
In \cite[Theorem 3.8]{block2005andre}, Block and Lazarev give an explicit formula when $f$ is \textbf{homotopy equivalent to a constant map}. They (re)proved that (under the convention that the cohomology is negatively graded)
$$\pi_{n}\Map(X,Y_{\mathbb{Q}})_{f}\cong \prod_{i=1}^{\infty}\pi_{i}(Y)\otimes \coh^{i-n}(X,{\mathbb{Q}}).$$ 
\subsection*{$p$-Adic homotopy theory}
When $p$ is a prime number and $k=\overline{\mathbb{F}}_{p}$ the algebraic closure of the field with $p$-elements. Mandell's fundamental theorem \cite{mandell2001sub} says that the $\infty$-category of $p$-complete spaces (with some finiteness conditions) is $\infty$-equivalent to a full $\infty$-subcategory of $\EE$-differential graded $k$-algebras via the cochain functor $C^{\ast}(-,\overline{\mathbb{F}}_{p})$. Suppose that $f:X\rightarrow Y$ is a map of simply connected spaces (with some finiteness conditions), then the forgetful functor $U$ induces the following map of abelian groups ($n>0$)
$$\pi_{n+1}\alpha:~\pi_{n+1}\Map(X,Y^{\wedge}_{p})_{f}\rightarrow \HH^{-n}(C^{\ast}(Y), C^{\ast}(X)_{f}),$$ 
where $Y\rightarrow Y^{\wedge}_{p}$ is a $p$-completion functor. 
\begin{theorem}[\ref{main}]
Suppose that $k=\mathbb{Q}$. Let $f: X\rightarrow Y$ be a map of simply connected spaces of (finite type), then the forgetful functor $$U:\Map_{\EE-\Algk}(C^{\ast}(Y),C^{\ast}(X))\rightarrow \Map_{\Algk}(C^{\ast}(Y),C^{\ast}(X))$$
induces a map of $k$-vector spaces such that:
\begin{enumerate}
\item $[X,Y_{\mathbb{Q}}]=\pi_{0}\Map_{\EE-\Algk}(C^{\ast}(Y),C^{\ast}(X))\rightarrow \pi_{0}\Map_{\Algk}(C^{\ast}(Y),C^{\ast}(X))$ is injective.
\item
$\pi_{1}\Map(X,Y_{\mathbb{Q}})_{f}=\pi_{1}\Map_{\EE-\Algk}(C^{\ast}(Y),C^{\ast}(X))_{f}\rightarrow \pi_{1}\Map_{\Algk}(C^{\ast}(Y),C^{\ast}(X))_{f}$ is injective map of groups.
\item $ \forall~n>0, $ the induced map
$$\pi_{n+1}\Map(X,Y_{\mathbb{Q}})_{f}=\pi_{n+1}\Map_{\EE-\Algk}(C^{\ast}(Y),C^{\ast}(X))_{f}\rightarrow \HH^{-n}(C^{\ast}(Y), C^{\ast}(X)_{f}),$$
 is injective map of $\mathbb{Q}$-vector spaces.
\item
If $X=Y$ and $f=id$, then $\pi_{1}Aut(X_{\mathbb{Q}})_{id}\rightarrow \HH^{0,\times}(C^{\ast}(X),C^{\ast}(X))$ is an injective map of abelian groups.
\end{enumerate}
 The space $X_{\mathbb{Q}}$ is the rationalization of $X$, $Aut(X)$ is the monoid of self equivalences, and 
$\HH^{0,\times}(C^{\ast}(X),C^{\ast}(X))$ is the group of invertible elements of the $k$-algebra $\HH^{0}(C^{\ast}(X),C^{\ast}(X)).$ 
\end{theorem}
\begin{warning}
When $k=\overline{\mathbb{F}}_{p}$, the induced maps $\pi_{n}\alpha$ are far to be injective in general.
\end{warning}
\begin{theorem}[Hodge filtration \ref{main2}]
With the same assemption as in precedent Theorem, we have the following isomorphism
$$\pi_{n+1}\Map(X,Y_\mathbb{Q})_{f}\cong  \mathrm{HH}^{-n}_{(1)}(C^{\ast}(Y),C^{\ast}(X)_{f}), ~\forall~n>0, \forall ~f.$$
\end{theorem}
\section{General framework}

For what follows we fixe $k=\mathbb{Q}$. Notice that $\EE-\Algk\simeq\CAlgk$. In the abstract we described only the applications. In order to prove them we pass by the model category of differential graded algebras (commutative and non-commutative). We denote the pointed model category of augmented (resp. commutative and $\EE$) differential graded $k$-algebras by $\Algk^{\ast}$ (resp. $\CAlgk^{\ast} $ and $\EE-\Algk^{\ast}$). Notice that the model structure in the commutative case make sense when $k$ is of characteristic 0. For some technical reasons, we define the functor of cochain complexes $C^{\ast}(-,k)=C^{\ast}(-):\sSet^{op}\rightarrow \EE-\Algk$. In this section, a space means a simplicial set.  
\begin{notation}\label{EE}
All differential graded algebras are non-negatively graded and the differentials increase the degree by +1.
Consider the map of operads (in the differential graded context) $\mathsf{Ass}\rightarrow \mathsf{Com}$, since the category of operad is a model category we have the factorization 
$\mathsf{Ass}\rightarrow\EE\rightarrow \mathsf{Com}$, where the first map is a cofibration and the second map is a trivial fibration. We have shown in \cite[Lemma 1.1]{amranirational}, that $\EE$ is admissible and the forgetful functor $U:\EE-\Algk^{\ast}\rightarrow \Algk^{\ast}$ preserves cofibrant objects and cofibration between cofibrant objects. That is the reason why we work with $\EE$-operad instead of the operad $\mathsf{Com}$.   
\end{notation}

Recall that we have a following diagram of (Quillen) adjunctions:

$$\xymatrix{\Algk\ar@<2pt>[rr]^{F}\ar@<2pt>[dd]^{\underline{U}}&& \EE-\Algk\ar@<2pt>[dd]^{\underline{U}}\ar@<2pt>[ll]^{U} \\
&&\\
\Algk^{\ast}\ar@<2pt>[rr]^{F}\ar@<2pt>[uu]^-{-\oplus k} && \EE-\Algk^{\ast}\ar@<2pt>[uu]^-{-\oplus k}\ar@<2pt>[ll]^{U}}$$
where, $F$ and $\underline{U}$ are left adjoints and $U$, $\oplus k$ are right adjoints.
 \begin{warning}
In what follows, we took the liberty to not denote the forgetful functors i.e., when $R$ is an (augmented) $\EE$-differential graded algebra we consider it also as an (augmented) associative differential graded algebra without mentioning the forgetful functor. 
\end{warning}
\begin{theorem}\cite[Theorem 3.1]{amranirational}\label{th}
 Let $k=\mathbb{Q}$, for any $R$ and $S$ augmented commutative differential graded $k$-algebras, the forgetful functor $U:\EE-\Algk^{\ast}\rightarrow \Algk^{\ast}$ induces a map $\alpha:~\Map_{\EE-\Algk^{\ast}}(R,S)\rightarrow \Map_{\Algk^{\ast}}(R,S)$ such that 
$$\pi_{n}\Map_{\EE-\Algk^{\ast}}(R,S)\rightarrow \pi_{n}\Map_{\Algk^{\ast}}(R,S)$$
is injective map of groups for $n>0$. Moreover, the map $\Omega\Map_{\EE-\Algk^{\ast}}(R,S)\rightarrow \Omega\Map_{\Algk^{\ast}}(R,S)$ has a functorial retract with respect to the target argument $S$ .  
\end{theorem}
For more details, we refere to \ref{complement}.

\begin{lemma}\label{commutes}
Let $k$ be any field.
The (derived) functor $C^{\ast}(-): \sSet^{op}\longrightarrow \EE\Algk$ commutes with homotopy limits. 
\end{lemma}

\begin{proof}
The functor $C^{\ast}$  has a left adjoint (cf \cite[Proposition 4.2 ]{mandell2001sub}), they form a Quillen pair. the homotopy limits in $\sSet^{op}$ are the homotopy colimits in  $\sSet$, it follows that for any diagram $J\rightarrow \sSet$ we have an isomorphism $C^{\ast}(\hocolim_{j\in J} X_{j})\cong \holim_{j\in J} C^{\ast}(X_{j})$ in the homotopy category $\Ho(\sSet)$. 
\end{proof}
\begin{notation}
We denote the simplicial sphere of dimension $n$ by $S^{n}$. 
\end{notation}

\begin{definition}\label{connected}
Let $R$ be an augmented $\EE$-differential graded $k$-algebra, we say that $R$ connected if $\pi_{0}\Map_{\Algk^{\ast}}(R, k\oplus k)=\pi_{0}\Map_{\EE-\Algk^{\ast}}(R, k\oplus k)=\ast.$

\end{definition}
\begin{lemma}\label{fund}
Let $X$ be a pointed connected simplicial set, and let $R\in\EE-\Algk^{\ast}$ be connected (cofibrant). Then the induced map by the forgetful functor 
$$\Map_{\EE-\Algk^{\ast}}(R,C^{\ast}(X))\rightarrow \Map_{\Algk^{\ast}}(R,C^{\ast}(X))$$
 has a functorial (depending on $X$) retract in $\Ho(\sSetp)$. 
\end{lemma}
\begin{proof}
We define two functors $\Psi, \Phi:\sSetp^{op}\rightarrow \sSetp$ as follows
\begin{itemize}
\item
$\Psi(X)=\Map_{\EE-\Algk^{\ast}}(R, C^{\ast}(X))$ and
\item $\Phi(X)=\Map_{\Algk^{\ast}}(R, C^{\ast}(X)).$
\end{itemize}
These functors verify the following properties
\begin{enumerate}
\item They send a weak equivalence $X\rightarrow Y$ to a weak equivalence since the functor $C^{\ast}(-)$ preserves weak between cofibrant objects and $\Map_{\EE-\Algk^{\ast}}(R,-)$, $\Map_{\EE-\Algk^{\ast}}(R,-)$
preserves weak equivalence between fibrant objects since $R$ is cofibrant as $\EE$-algebra and as associative algebra cf \ref{EE}. 
\item The functors $\Psi$ and $\Phi$ take homotopy limits to homotopy colimits, it follows that the mapping spaces of a model category commutes with homotopy limits in the second argument and the fact that $C^{\ast}(-)$ takes homotopy colimits to homotopy limits \ref{commutes}. Moreover the forgetful functor $U:\EE-\Algk^{\ast}\rightarrow \Algk^{\ast}$ commutes with homotopy limites.  
\item $\Psi(\ast)$ and $\Phi(\ast)$ are contractible since $k$ is a terminal object in $\EE-\Algk^{\ast}$ and $\Algk^{\ast}.$
\end{enumerate}
It follows form \cite[Theorem 16]{jardine2009representability}, that $\Psi(-)$ and $\Phi(-)$ are representable i.e., there exists two simplicial sets  $C$ and $A$ such that 
$\Psi(-)\simeq \Map_{\sSetp}(-,C)$ and $\Phi(-)\simeq \Map_{\sSetp}(-,A)$ in $\Ho(\sSetp)$ the natural transformation $\Psi(-)\rightarrow \Phi(-)$ is represented by a map $C\rightarrow A$.
By theorem \ref{th}, we know that the map $\Omega\Psi(-)\rightarrow \Omega\Phi(-)$ has a functorial retract (in $\Ho(\sSetp)$), it follows that the map $\Omega\Map_{\sSetp}(-,C)\rightarrow \Omega\Map_{\sSetp}(-,A)$ has a functorial retract, it implies that $\Omega C\rightarrow \Omega A$ has a retract. On another hand $R$ is connected, it follows that $A$ and $C$ are connected. hence, the induced map $A\rightarrow C$ has a retract  in $\Ho(\sSet_{\ast})$. We conclude that 
$\Map_{\EE-\Algk^{\ast}}(R,C^{\ast}(X))\rightarrow \Map_{\Algk^{\ast}}(R,C^{\ast}(X))$ has a functorial retract in $\Ho(\sSetp)$ for any simplicial set $X$. 
\end{proof}
\begin{corollary}\label{groups}
For any connected augmented $\EE$-differential graded algebra, and any pointed simplicial set $X$, the natural map 
$$\Map_{\EE-\Algk^{\ast}}(R,C^{\ast}(X))\rightarrow \Map_{\Algk^{\ast}}(R,C^{\ast}(X))$$
induces an injective map on homotopy groups. 
\end{corollary}

\begin{theorem}\label{tit}
Let $R$ be a connected augmented $\EE$-differential graded algebra, with augmentation $\nu:R\rightarrow k$. Let $X$ be any pointed simplicial set, let $f:R\rightarrow C^{\ast}(X)$ be any map of augmented $\EE$-differential graded algebras. Then the induced map by the forgetful functor 
$$\alpha:~ \Map_{\EE-\Algk}(R,C^{\ast}(X))\rightarrow \Map_{\Algk}(R,C^{\ast}(X))$$
has a functorial retract (on the variable $X$), in particular  $\forall~f, \forall ~n> 0$: 
\begin{itemize}
\item
$\pi_{0}\alpha:\pi_{0} \Map_{\EE-\Algk}(R,C^{\ast}(X))\rightarrow \pi_{0} \Map_{\Algk}(R,C^{\ast}(X))$ and 
\item
$ \pi_{n}\alpha:\pi_{n} \Map_{\EE-\Algk}(R,C^{\ast}(X))_{f}\rightarrow \pi_{n} \Map_{\Algk}(R,C^{\ast}(X))_{f} ~$
\end{itemize}
are injective maps
\end{theorem}
\begin{proof}
First of all, notice that we have an obvious cofiber sequence of pointed simplicial sets 
$$\xymatrix{S^{0}\ar[r]^{i} & X_{+}\ar[r]^{p} & X}$$
where $X_{+}$ is the pointed simplicial set $X\coprod {\ast}$. It is enough to notice that 
$$\Map_{\EE-\Algk^{\ast}}(R,C^{\ast}(X_{+}))\simeq \Map_{\EE-\Algk}(R,C^{\ast}(X))$$
 and 
$$\Map_{\Algk^{\ast}}(R,C^{\ast}(X_{+}))\simeq \Map_{\Algk}(R,C^{\ast}(X)),$$
then the result follows from \ref{fund}. 

\end{proof}

\section{Main theorems and applications}
\begin{proposition}
Suppose that $k=\mathbb{Q}$. Let $R$ be an augmented $\EE$-differential graded $k$-algebra of finite type (i.e $\textrm{dim}_{k}\coh^{i}(R)<\infty~ \forall~ i$) such that $\coh^{0}(R)=k$ and $\coh^{1}(R)=0$  then $R$ is connected in the sense of \ref{connected}.
\end{proposition}
\begin{proof}
First of all, by adjunction $\Map_{\EE-\Algk^{\ast}}(R,k\oplus k)\simeq  \Map_{\EE-\Algk}(R,k)$. 
Without loosing generality we can suppose that $R$ is cofibrant as $\EE-\Algk$, hence $R$ is cofibrant as $\Algk$ (by construction of the operad $\EE$ cf \ref{EE}). By Sullivan Theorem $\pi_{0}\Map_{\EE-\Algk}(R,k)=\ast$. It follows that for any maps $\nu: R\rightarrow k$ and $\mu: R\rightarrow k$ are homotopic in $\EE-\Algk$. According to \cite{hinich1997homological}, we have a commutative diagram in $\EE-\Algk$
$$\xymatrix{
& k\\
R\ar[ur]^{\mu}\ar[dr]_{\nu}\ar[r]& P(R)\ar[u]\ar[d]\\
& k 
}$$
where $P(R)$ is a path object associated to $R$. Notice that the path object is the same for graded differential associative algebras if we consider $R\in\Algk$. Since $\coh^{0}(R)=k$ any map $R\rightarrow k$ in $\Algk$ is actually a map in $\EE-\Algk$. We conclude that 
$\pi_{0}\Map_{\Algk}(R,k)=\ast=\pi_{0}\Map_{\Algk^{\ast}}(R,k\oplus k).$

\end{proof}
\begin{theorem}[Main Theorem]\label{main}
Suppose that $k=\mathbb{Q}$. Let $f: X\rightarrow Y$ be a map of simply connected spaces of (finite type), then the forgetful functor $$U:\Map_{\EE-\Algk}(C^{\ast}(Y),C^{\ast}(X))\rightarrow \Map_{\Algk}(C^{\ast}(Y),C^{\ast}(X))$$
induces a map of $k$-vector spaces such that:
\begin{enumerate}
\item $[X,Y_{\mathbb{Q}}]=\pi_{0}\Map_{\EE-\Algk}(C^{\ast}(Y),C^{\ast}(X))\rightarrow \pi_{0}\Map_{\Algk}(C^{\ast}(Y),C^{\ast}(X))$ is injective.
\item
$\pi_{1}\Map(X,Y_{\mathbb{Q}})_{f}=\pi_{1}\Map_{\EE-\Algk}(C^{\ast}(Y),C^{\ast}(X))_{f}\rightarrow \pi_{1}\Map_{\Algk}(C^{\ast}(Y),C^{\ast}(X))_{f}$ is injective.
\item $ \forall~n>0, $ 
$$\pi_{n+1}\Map(X,Y_{\mathbb{Q}})_{f}=\pi_{n+1}\Map_{\EE-\Algk}(C^{\ast}(Y),C^{\ast}(X))_{f}=\mathrm{AQ}^{-n-1}(C^{\ast}(Y),C^{\ast}(X)_{f})$$ and the induced map
$$\pi_{n+1}\Map(X,Y_{\mathbb{Q}})_{f}\rightarrow \pi_{n+1}\Map_{\EE-\Algk}(C^{\ast}(Y),C^{\ast}(X))_{f}=\HH^{-n}(C^{\ast}(Y), C^{\ast}(X)_{f})$$
 is an injective map of $\mathbb{Q}$-vector spaces.
\item
If $X=Y$ and $f=id$, then $\pi_{1}Aut(X_{\mathbb{Q}})_{id}\rightarrow \HH^{0,\times}(C^{\ast}(X),C^{\ast}(X))$ is an injective map of abelian groups.
\end{enumerate}

\end{theorem}
\begin{proof}
By hypothesis $X$ and $Y$ are of finite type, 
we deduce by \cite{sullivan1977infinitesimal} that $$\Map_{\EE-\Algk}(C^{\ast}(Y),C^{\ast}(X))$$ is equivalent to $\Map(X,Y_{\mathbb{Q}})$, on the other hand by Theorem \ref{tit}, the forgetful functor $U:\EE-\Algk\rightarrow \Algk$ induces an injective map $$\alpha:\pi_{i}\Map_{\EE-\Algk}(C^{\ast}(Y),C^{\ast}(X))_{f}\rightarrow \pi_{i}\Map_{\Algk}(C^{\ast}(Y),C^{\ast}(X))_{f}$$ for all $i\geq 0$. Moreover if $i>1$, Block-Lazarev theorem gives us the isomorphism  $$\pi_{i}\Map_{\EE-\Algk}(C^{\ast}(Y),C^{\ast}(X))_{f}\cong \mathrm{AQ}^{-i}(C^{\ast}(Y),C^{\ast}(X)_{f}) ,$$
and by \cite{amrani2013mapping}, $$\pi_{i}\Map_{\Algk}(C^{\ast}(Y),C^{\ast}(X))_{f}\cong  \HH^{-i+1}(C^{\ast}(Y),C^{\ast}(X)_{f})$$
Hence, the induced map $\alpha$ is exactly 
 $\mathrm{AQ}^{-i}(C^{\ast}(Y),C^{\ast}(X)_{f})\rightarrow \HH^{-i+1}(C^{\ast}(Y),C^{\ast}(X)_{f})$, which is injective map of $\mathbb{Q}$-vector spaces. Applying Sullivan theorem we deduce that  $\pi_{i}\Map(X,Y_{\mathbb{Q}})_{f}\cong \mathrm{AQ}^{-i}(C^{\ast}(Y),C^{\ast}(X)_{f})$ for $i>1$ . 
 In particular, when $X=Y$ and $f=id$, $\Map(X,X)_{id}=Aut(X)_{id}$ and 
 $$\Map(X_{\mathbb{Q}},X_{\mathbb{Q}})_{id}= Aut(X_{\mathbb{Q}})_{id}.$$
Therefore, $\pi_{1}Aut(X_{\mathbb{Q}})_{id}\cong \pi_{1}\Map_{\EE-\Algk}(C^{\ast}(X),C^{\ast}(X))_{id}$. In \cite[Corollary 3.6]{amrani2013mapping}, we have shown that $\pi_{1}\Map_{\Algk}(C^{\ast}(X),C^{\ast}(X))_{id}$ is isomorphic to the kernel of the natural map of (abelian) groups $\HH^{0,\times}(C^{\ast}(X),C^{\ast}(X))\rightarrow \mathrm{H}^{0, \times}(C^{\ast}(X))=\mathbb{Q}^{\times}$. The result follows for Theorem \ref{tit}. 
\end{proof}

\begin{corollary}\label{loop}
Let $M$ be a simply connected orientable closed manifold of dimension $d$, for all $i> 0$,  we have an injective map of $\mathbb{Q}$ vector spaces 
$$\pi_{i}\Omega_{id}Aut(M)\otimes \mathbb{Q}\rightarrow \mathrm{H}_{i+d}(\mathcal{L}M,\mathbb{Q}),$$ 
where $\mathcal{L}M$ is the space of free loops on $M$, i.e., $\Map(S^{1},M)$.  
\end{corollary}
\begin{proof}
Since $M$ is a finite CW-complex, it is a direct consequence of Theorem \ref{main}, and the fact that \\
 $\HH^{\ast}(C^{\ast}(M),C^{\ast}(M))\cong \mathrm{H}_{\ast+d}(\mathcal{L}M,\mathbb{Q})$ \cite{cohen2002homotopy}.
\end{proof}
\begin{remark}
Corollary \ref{loop} was also proven in \cite[Theorem 2 (1)]{felix2004monoid} using a different method.  
\end{remark}
\subsection{Hodge filtration on Hochschild cohomology over a field of characteristic zero } In our main Theorem \ref{main}, we have identified the higher homotopy groups of $\Map(X,Y_{\mathbb{Q}})_{f}$ based at some continuous map $f:X\rightarrow Y$ as a sub $\mathbb{Q}$-vector space of the (negative) Hochschild cohomology. According to \cite[Theorem 3.1]{ginot2010hochschild}, there exists a Hodge decomposition on the Hochschild cohomology $\HH^{\ast}(R,S)$ for any differential graded $\mathbb{Q}$-algebra $R$ and any differential graded $R$-bimodule $S$.  More precisely Ginot has proved in \cite{ginot2010hochschild}, the following formula in the rational case:
$$\mathrm{HH}^{\ast}(R,S)\cong \prod_{n\geq 0} \mathrm{HH}^{\ast}_{(n)}(R,S),$$
where the $\mathbb{Q}$-vector spaces $ \mathrm{HH}^{\ast}_{(n)}(R,S)$ are eigenspaces for an iterated power of some operator.  

\begin{theorem}\label{main2}
With the same assemption as in Theorem \ref{main}, we have the following isomorphism
$$\pi_{n+1}\Map(X,Y_{\mathbb{Q}})_{f}\cong  \mathrm{HH}^{-n}_{(1)}(C^{\ast}(Y),C^{\ast}(X)_{f}), ~\forall~n>0, \forall ~f.$$
\end{theorem}
\begin{proof}
First of all, we notice that $\pi_{n}\Map(X,Y_{\mathbb{Q}})_{f}\cong \mathrm{AQ}^{-n}(C^{\ast}(Y),C^{\ast}(X)_{f})$ for all $n>1$ (cf \cite{block2005andre}), where $\mathrm{AQ}^{\ast}$ is the Andr\'e-Quillen cohomology. On another hand $ \mathrm{HH}^{-n}_{(1)}(C^{\ast}(Y),C^{\ast}(X)_{f})=\mathrm{Harr}^{-n}(C^{\ast}(Y),C^{\ast}(X)_{f})$, where $\mathrm{Harr}^{\ast}$ is the Harrison cohomology, cf  \cite[Theorem 3.1]{ginot2010hochschild}. Since we work in characteristic zero, Harrison cohomology and Andr\'e-Quillen cohomology agree up to a shift, more precisely $\mathrm{AQ}^{n-1}=\mathrm{Harr}^{n}$. It follows that 

\begin{eqnarray*}
\pi_{n+1}\Map(X,Y_{\mathbb{Q}})_{f}&\cong& \mathrm{AQ}^{-n-1}(C^{\ast}(Y),C^{\ast}(X)_{f})\\
&\cong& \mathrm{Harr}^{-n}(C^{\ast}(Y),C^{\ast}(X)_{f})\\
&\cong&  \mathrm{HH}^{-n}_{(1)}(C^{\ast}(Y),C^{\ast}(X)_{f}),~ \forall~n>0, \forall ~f. 
\end{eqnarray*}
\end{proof}
\begin{remark}
Theorem \ref{main2} is a generalization of \cite[Theorem 2 (2)]{felix2004monoid}.
\end{remark}

\section*{Appendix}
There is a class of model categories called simplicial model categories \cite{goerss1999}, roughly speaking a simplicial model category is tensored, cotensored and enriched over the model category of simplicial sets in a compatible way (adjunction compatibility, and model structure compatibility). In general a model category $\C$ do not need to be simplicial model category. Moreover, a Quillen adjunction between simplicial model categories 
$$\xymatrix{ \C \ar@<2pt>[r]^{F} & \D, \ar@<2pt>[l]^{U}  }$$ 
is not a simplicial adjunction in general. 
In \cite[Chapter 5, 6 ]{Hovey}, Hovey introduced a notion of module category. We will need a more richer structure and we will call it \textbf{enriched module structure}. In the classical context any ordinary category with product and coproduct is an enriched $\Set$-module. More precisely, 
suppose that $\D$ is an ordinary category with products and coproducts, we can define the following functors:
\begin{enumerate}
\item
$-\otimes-:~\Set\times \D\rightarrow \D$ such that for any set $X$ and any object $D\in \D$ we have $X\otimes D=\coprod_{i\in X} D$.
\item
$A(-,-):~\Set^{op}\times \D\rightarrow \D$ such that for any set $X$ and for any $D\in \D$ we define $A(X,D)=\prod_{i\in X}D$.
\end{enumerate}
\begin{definition}\label{def1}
An enriched $\Set$-module $\D$ is a category with all products and coproducts such that we have natural isomorphism for any $X, Y\in \Set$ and any $C, D\in \D$
\begin{itemize}
\item $(X\times Y)\otimes D\cong X\otimes (Y\otimes D)$.
\item $\hom_{\D}(C,A(X,D))\cong \hom_{\D}(X\otimes C,D)$.
\item $\hom_{\Set}(X,\hom_{\D}(C,D))\cong \hom_{\D}(X\otimes C, D)\cong \hom_{\D}( C, A(X,D)).$
\item $\ast\otimes D\cong D$. 
\end{itemize}
\end{definition}

A simplicial category $\D$ in the sense of  \cite{goerss1999} is an enriched $\sSet$-module in the sense of \ref{def1}, where we replace $\hom_{\D}$ by the natural enrichment of $\D$ denoted by $\Map_{\D}$ (simplicial set). 
\begin{theorem}\label{enrich}
Let $\D$ be any (pointed) model category, then the homotopy category $\Ho(\D)$ is an enriched $\Ho(\sSet)$-module (enriched $\Ho(\sSetp)$-module).
\end{theorem}
\begin{proposition}
Given any Quillen adjunction between model categories 
$$\xymatrix{ \C \ar@<2pt>[r]^{F} & \D, \ar@<2pt>[l]^{U}  }$$ 
it induces the following isomorphisms :
\begin{itemize}
\item $\Map_{\C}(X,U(Y))\cong\Map_{\D}(F(X),Y)$ in $\Ho(\sSet)$
\item  $ F(X\otimes C)\cong X\otimes F(C)$  in $\Ho(\D)$ for any $X\in\sSet$ and any $C\in \C$. 
\item $U (A(X,D))\cong A(X,UD)$ in $\Ho(\C)$ for all $X\in\sSet$ and any $D\in \D$
\end{itemize}
\end{proposition} 
The proof of the precedent theorem and proposition can be deduced from \cite{Hovey}. 
The involved mapping spaces tensors and cotensors are defined in the \textbf{derived} sense, we took the liberty to not specify the derived symbols (e.g. $\mathbf{R}$ and $\mathbf{L}$).  

\begin{notation} If $\D$ is a pointed model category, we denote by $\Omega D$ the object $A(S^{1},D)$
and $\Sigma D$ the object $S^{1}\otimes D$.  
\subsection{Complement to Theorem \ref{th}}\label{complement}
We explain, the cited Theorem using the previous language. Let $R$ be cofibrant an augmented $\EE$-differential graded $\mathbb{Q}$-algebras. Considering the adjunction 
$$\xymatrix{ \Algk^{\ast} \ar@<2pt>[r]^-{F} & \EE-\Algk^{\ast}, \ar@<2pt>[l]^-{U}}$$ 
our theorem says that we have a natural map $S^{1}\otimes R\rightarrow F(S^{1}\otimes UR)\simeq S^{1}\otimes FUR $ which has a retract in $\Ho(\EE-\Algk^{\ast})$. In other words, suppose that $S\in\EE-\Algk^{\ast}$,  we have a retract in $\Ho(\sSetp)$ of the map
$$\Map_{\EE-\Algk^{\ast}} ( S^{1}\otimes R,S)\rightarrow \Map_{\EE-\Algk^{\ast}}( S^{1}\otimes FUR,S)$$
which can be rewritten by using adjunctions as:
$$h:\Omega \Map_{\EE-\Algk^{\ast}} (  R,S)\rightarrow \Omega\Map_{\Algk^{\ast}}(UR,US),$$
such that, there is an induced left inverse map $r$, i.e., $r\circ h=id$ and it is functorial with respect to $S$.  
\end{notation}


\bibliographystyle{plain} 
\bibliography{inclusion}

\end{document}